\documentclass[12pt]{amsart}
\usepackage{extarrows}
\usepackage{amssymb}
\usepackage{amsmath,latexsym}
\usepackage{amsfonts}
\usepackage{amsthm}
\usepackage{eucal}
\usepackage{color}
\usepackage{todonotes}

\allowdisplaybreaks[4]

\newcommand{\p}{\partial}

\usepackage{enumerate}
\usepackage{mathrsfs}
\usepackage{mathtools}
\theoremstyle{definition}
\newtheorem{definition}{Definition}
\newtheorem{theorem}{Theorem}
\newtheorem{lemma}[theorem]{Lemma}
\newtheorem{proposition}[theorem]{Proposition}
\newtheorem{corollary}[theorem]{Corollary}
\newtheorem{remark}[theorem]{Remark}

\usepackage[colorlinks,linkcolor=blue]{hyperref}
\begin{document}
\title{Compacitification and positive mass theorem \\
for fibered Euclidean end }
\author{Xianzhe Dai}
\address{
Department of Mathematics, 
University of Californai, Santa Barbara
CA93106,
USA}
\author{Yukai Sun}
\address{
School of Mathematics Sciences, East China Normal University
Shanghai, China}
\date{\today}

\begin{abstract}
In this note, we consider the positive mass theorem for Riemannian manifolds $(M^{n},g)$ asymptotic to $(\mathbb{R}^{k}\times X^{n-k}, g_{\mathbb{R}^{k}}+g_{X})$ for $k\geq 3$ by studying the corresponding compactification problem.
\end{abstract}

\maketitle

\section{Introduction}

The famous Positive Mass Theorem \cite{SY, W} states that an asymptotically Euclidean manifold with nonnegative scalar curvature must have nonnegative ADM mass (if the dimension of the manifold is between $3$ and $7$ or if the manifold is spin; for the recent progress about the higher dimensional non-spin manifolds, see the preprint \cite{SY1}). Furthermore, the mass is zero iff the manifold is the Euclidean space.

A crucial ingredient in Schoen-Yau's approach (after Lohkamp's observation) is the idea of compactification. In the end, the Positive Mass Theorem for $M$ is reduced to showing that, for the one-point compactification $M_1$ of $M$, the connected sum $M^{n}_{1}\# T^n$ admits no positive scalar curvature metric. Roughly speaking, one compactifies the asymptotically Euclidean manifold $M^{n}$ to $M^{n}_{1}\# T^{n}$.

In \cite{D}, motivated by string theory, the first author proves a Positive Mass Theorem for manifolds which are asymptically approaching the product of an Euclidean space and a compact Calabi-Yau (or any exceptional holonomy) manifolds. The approach there is along that of Witten, using Dirac operators, and hence our manifolds are assumed to be spin. More crucially, one needs nonzero parallel spinors for the asymptotic infinity which limits the geometry at infinity. For technical reasons, we restricts to simply connected compact factors, although the method extends to non-simply connected case with some additional assumptions on the spin structure. This was later considered by \cite{Minerbe1} who treated the case of $S^1$ factor (Minerbe is motivated by the study of gravitational instantons; he also allows non-trivial circle fibrations). 
More recently, Positive Mass Theorems for the $S^1$ or $T^2$ factors have been considered in \cite{LiuShiZhu}, and more generally flat factors in \cite{CLSZ}. 


More precisely, a complete Riemannian manifolds $(M^{n},g)$ is said to have asymptotically fibered Euclidean end if $M^{n}=M_{0}\cup M_{\infty}$ with $M_{0}$ compact and $M_{\infty}\simeq(\mathbb{R}^{k}\backslash B_{R}(0))\times X (k\geq 3)$ for some $R>0$ and $X$ a close $n-k$ dimensional manifold. Moreover,  the metric on $M_{\infty}$ satisfies
\begin{eqnarray*}
  g&=&\hat{g}+h,\quad \hat{g}=g_{\mathbb{R}^{k}}+g_{X},\\
  h&=&O(r^{-\tau}),\;\hat{\nabla}h=O(r^{-\tau-1}),\; \hat{\nabla}\hat{\nabla}h= O(r^{-\tau-2}),\; \hat{\nabla}\hat{\nabla}\hat{\nabla}h= O(r^{-\tau-3})
  \end{eqnarray*}
where $\hat{\nabla}$ is the Levi-Civita connection of $\hat{g}$, $\tau>\frac{k-2}{2}$ is the asymptotical order and $r$ is the Euclidean distance to a point. We will also call such $M^{n}$ a manifold with asymptotic end $\mathbb{R}^{k}\times X$.

 The mass for such a space is then defined by \cite{D}
 $$m(g)=\lim_{R\to\infty}\frac{1}{\omega_{k}\operatorname{Vol}(X)}\int_{S_{R}\times X}(\hat{\nabla}_{e_{a}}g_{ja}-\hat{\nabla}_{j}g_{aa})\ast dx_{j}d\operatorname{vol}_{g_{X}}.$$
 where $\{e_{a}\}=\{\frac{\partial}{\partial x_{i}},f_{\alpha}\}$ is an orthornormal basis of $\hat{g}$ and the $\ast$ operator is the  one on the Euclidean factor, the indices $i,j$ run over the Euclidean factor and the index $\alpha$ runs over $X$ while the index $a$ runs over the full index of the manifold $M^{n}$. (There is an additional factor of $\frac{1}{4}$ in \cite{D}.)

Strictly speaking, the existence of the limit in the definition of the mass is only guaranteed under additional assumptions
such as the integrability of the scalar curvature \cite{B}. (More generally, one replaces the limit by $\lim \sup$.) For this reason (and some others) we now assume that the metric $g_X$ on $X$ has zero scalar curvature. 

On the other hand, if $M^n_1$ is a closed $n$-dimensional manifold and $X$ embeds in $M^n_1$ with trivial normal bundle, then one can construct the connected sum of $M^n_1$ and $T^k \times X$ along $X$,
$M^{n}_{1}\#_{X}T^{k}\times X$. Our main result is

\begin{theorem}\label{maintheorem}
  If for any closed manifold $M^{n}_{1}$ admitting an embedding of $X$ with trivial normal bundle, the connected sum along $X$ $$M^{n}_{1}\#_{X}T^{k}\times X$$ has no metrics of positive scalar curvature, then the mass of any manifold  $M^{n}$ with asymptotic end $\mathbb{R}^{k}\times X$ and nonnegative scalar curvature is nonnegative.
  Moreover, $(M,g)\equiv (\mathbb{R}^{k}\times X,g_{\mathbb{R}^{k}}+g_{X})$  if the mass of $(M^{n},g)$ is zero.
\end{theorem}

As in \cite{1}, the first part of Theorem \ref{maintheorem} is a consequence of the following result. Recall that if $M^{n}$ is a manifold with asymptotic end $\mathbb{R}^{k}\times X$, then $M^{n}$ decomposes as $M_{0}\cup M_{\infty}$ with $M_{0}$ compact and $M_{\infty}\simeq(\mathbb{R}^{k}\backslash B_{R}(0))\times X$ for some $R>0$. In particular $\p M_0\simeq S^{k-1} \times X$. 

\begin{theorem} \label{comp}
  Let $(M^{n}, g)$ be a manifold with asymptotic end $\mathbb{R}^{k}\times X$ and $M^n_1$ the closed manifold obtained from $M_0$ by attaching $B^k \times X$, where $B^k$ is the Euclidean $k$-ball.
  If the scalar curvature of $M^{n}$ is nonnegative but its mass is negative, then the connected sum along $X$ $$M^{n}_{1}\#_{X}T^{k}\times X$$ admits a metrics of positive scalar curvature. 
\end{theorem}

As application of Theorem \ref{maintheorem}, we have
\begin{corollary}
  Assume that either $X$ is enlargeable or $\hat{A}(X)$ is nonzero. Then the Positive Mass Theorem holds for any spin manifolds with asymptotic end $\mathbb{R}^{k}\times X$.
\end{corollary}
For definitions of the enlargeability and the $\hat{A}$-genus $\hat{A}(X)$ see Section $5$ of \cite{6}.

As we assume $(X, g_X)$ is scalar flat, if $X$ is further assumed to be simply connected and $\dim X \geq 5$, then in fact $(X, g_X)$ is the product of Calabi-Yau manifolds or $Spin(7)$-manifolds by Theorem 1 in \cite{F}. Thus we give another proof of the result in \cite{D}, except for the $G_2$-manifold factors. Simply connected $G_2$-manifolds always admit positive scalar curvature metrics by \cite{St}. There lies the difference between the approach in \cite{D} and the current approach. We emphasize that our result for the non-simply connected case is new.


The proof of the compacification theorem, Theorem \ref{comp}, consists of two steps. By solving a conformal Laplacian equation, one can deform the metric $g$ outside a compact set to a conformal metric of the product metric $g_{\mathbb{R}^{k}}+g_{X}$, while still maintaining negative mass and good scalar curvature control. The crucial thing here is the asymptotic behavior of the solution at infinity, which is obtained by studying the asymptotics of the Green's function of $\Delta_g$. In the second step, which is due to an observation of Lohkamp, one modifies the conformal factor by a subharmonic function (again outside a compact set) which is constant near infinity. Again the asymptotic behavior of the conformal factor is critical.

The rigidity part of the Positive Mass Theorem makes use of two mass notions and a result in \cite{CLSZ}, generalized to our situation. The so called Gauss-Bonnet mass is introduced in \cite{Minerbe1} for the case $S^1$ factor and it in general is different from the mass discussed above. But as in \cite{CLSZ}, under certain geometric conditions which is satisfied when the mass is zero, the two masses differ by a nonzero multiple. This is crucial in showing that the harmonic functions coming from the asymptotic coordinates have parallel differentials. We then use these functions with parallel differential to directly construct the required isometry. In \cite{CLSZ} the authors also obtain rigidity for the $S^1$ factor and more generally flat factors, although the conclusion is weaker for the general flat factors.

This paper is organized as follows.  In Section \ref{prelim}, we introduce manifolds with asymptotic end $\mathbb{R}^{k}\times X$ and their topological compactifications. For later purpose we also discuss the existence of 
positive scalar curvature metrics on such manifolds (and their generalized connected sums). In Section \ref{Green}, we study the asymptotic behavior at infinity of the Green function of the metric $g$. 
In Section \ref{compact}, we solve the conformal Laplace equation and use it to prove the Theorem \ref{compact}. Finally in Section \ref{rigidity} we clarify the roles of the two mass quantities and establish the rigidity
statement in Theorem \ref{maintheorem}.

\section{Manifolds with asymptotic end $\mathbb{R}^{k}\times X$} \label{prelim}

In this section we give the precise definition of manifolds with asymptotic end $\mathbb{R}^{k}\times X$ and its mass(es). Let $g_{\mathbb{R}^{k}}$ be the standard Euclidean metric. Let $(X, g_X)$ be a compact Riemannian manifold.
\begin{definition}
  A complete noncompact smooth manifold $(M^{n},g)$ is called asympototically $(\mathbb{R}^{k}\times X^{n-k}, g_{\mathbb{R}^{k}}+g_{X})$ for $k\geq 3$, if there exists a compact subset $K\subset M^{n}$ such that $M^{n}\backslash K\simeq (\mathbb{R}^{k}-B_{R}(0))\times X $ for some $R>0$ and on $M^{n}\backslash K$
  \begin{eqnarray*}
  g&=&\hat{g}+h,\quad \hat{g}=g_{\mathbb{R}^{k}}+g_{X},\\
  h&=&O(r^{-\tau}),\;\hat{\nabla}h=O(r^{-\tau-1}),\; \hat{\nabla}\hat{\nabla}h= O(r^{-\tau-2}),\; \hat{\nabla}\hat{\nabla}\hat{\nabla}h= O(r^{-\tau-3})
  \end{eqnarray*}
    where $\hat{\nabla}$ is the Levi-Civita connection of $\hat{g}$, $\tau>\frac{k-2}{2}$ is the asymptotical order. In addition, we assume that $Sc_{g}\in L^{1}(M)$.
\end{definition}

As we mentioned, the assumption $Sc_{g}\in L^{1}(M)$ is imposed so that the ADM mass is well defined \cite{B}. In particular this would imply that the scalar curvature of $(X, g_X)$ is identically zero. From now on, we make the assumption that $Sc_{g_{X}}\equiv 0$. Let $\omega_{k}$ be the area of the unit sphere in $\mathbb{R}^{k}$.
\begin{definition}\label{massdef}
   The mass for such a space is then defined by \cite{D}
 $$m(g)=\lim_{R\to\infty}\frac{1}{\omega_{k}\operatorname{Vol}(X)}\int_{S_{R}\times X}(\hat{\nabla}_{e_{a}}g_{ja}-\hat{\nabla}_{j}g_{aa})\ast dx_{j}d\operatorname{vol}_{g_{X}}.$$
 where $\{e_{a}\}=\{\frac{\partial}{\partial x_{i}},f_{\alpha}\}$ is an orthornormal basis of $\hat{g}$ and the $\ast$ operator is the  one on the Euclidean factor, the indices $i,j$ run over the Euclidean factor and the index $\alpha$ runs over $X$ while the index $a$ runs over the full index of the manifold $M^{n}$.
\end{definition}

The same argument as in \cite{B} shows that $m(g)$ is well-defined metric invariant of $(M, g)$.
It reduces to, by the Fubini's theorem and divergence theorem,
\begin{eqnarray*}
  m(g) &=& \lim_{R\to\infty}\frac{1}{\omega_{k}\operatorname{Vol}(X)}\int_{S_{R}\times X}(\partial_{i}g_{ij}-\partial_{j}g_{aa})\ast dx_{j}d\operatorname{vol}_{g_{X}}.
\end{eqnarray*}

More intrinsically, it can be written as \cite{Minerbe1}
\begin{eqnarray}
m(g)=-\frac{1}{\omega_{k}\operatorname{Vol}_{X}}\lim_{R\to \infty} \int_{\partial B_{R}} *_{\hat{g}}(\operatorname{div}_{\hat{g}}g+d\operatorname{Tr}_{\hat{g}}g).
\end{eqnarray}
Here $*_{\hat{g}}$ is the Hodge star operator of the metric $\hat{g}$, $\operatorname{Tr}_{\hat{g}}$ indicates taking trace with respect to $\hat{g}$, and $(\operatorname{div}_{\hat{g}}T)_{i}=-\hat{g}^{jk}\partial_{j}T_{ik}$ in local coordinate for a $(0,2)$-type tensor $T$.

\begin{remark}
Without the integrability assumption $Sc_{g}\in L^{1}(M)$, one can still define the ADM mass
as above, except replacing the limit by $\lim \sup$ \cite{Minerbe1}.
\end{remark}

We now define another mass quantity, the so called Gauss-Bonnet mass, first introduced in \cite[Theorem 1]{Minerbe1} for the circle fibration. 
\begin{eqnarray}
m^{GB}(g)=-\frac{1}{\omega_{k}\operatorname{Vol}_{X}}\lim_{R\to \infty}\sup\int_{\partial B_{R}} *_{\hat{g}}(\operatorname{div}_{\hat{g}}g+d\operatorname{Tr}_{\hat{g}}g-\frac{1}{2}d(\operatorname{Tr}_{g_{X}} g)).
\end{eqnarray}
Here $\operatorname{Tr}_{g_X}$ indicates taking trace with respect to $g_X$. In other words, $\operatorname{Tr}_{g_{X}} g=\sum_{\alpha} g(f_{\alpha}, f_\alpha)$ for an orthonormal basis $f_{\alpha}$ of $(X, g_X)$. This mass quantity will play an important part in proving the rigidity part of Theorem \ref{maintheorem}; see Section \ref{rigidity}.

A manifold $(M, g)$ with asymptotic end $\mathbb{R}^{k}\times X$ has a natural topological compactification. 
Indeed, $M^{n}$ decomposes as $M_{0}\cup M_{\infty}$ with $M_{0}$ compact and $M_{\infty}\simeq(\mathbb{R}^{k}\backslash B_{R}(0))\times X$ for some $R>0$. In particular $\p M_0\simeq S^{k-1} \times X$. 
One can then ontain a closed manifold $M^n_1$ by attaching $B^k \times X$ to $M_0$, where $B^k$ is the Euclidean $k$-ball.
This is in some sense the fiberwise version of the one-point compactification for the asymptotically Euclidean spaces.

Note that $X$ embeds into $M_1$ with trivial normal bundle. In general, if $M^n_1$ is a closed $n$-dimensional manifold and $X$ embeds in $M^n_1$ with trivial normal bundle, then one can construct the connected sum of $M^n_1$ and $T^k \times X$ along $X$, denoted by $M^{n}_{1}\#_{X}T^{k}\times X$. In other words, $M^{n}_{1}\#_{X}T^{k}\times X$ is obtained from $M^n_1$ by removing a tubular neighborhood of $X$ and gluiing in $(T^k\backslash B_\epsilon(x))\times X$, where $x$ is a fixed point on $T^k$ and $B_\epsilon(x)$ an $\epsilon$-ball around $x$.

To consider applications of Theorem \ref{maintheorem}, we now look at some examples of manifolds of the type $M_{1}^{n}\#_{X}T^{k}\times X^{n-k}$ admitting no positive scalar curvature. 
For definitions of the enlargeability and the $\hat{A}(X)$ see Section $5$ in \cite{6}.

\begin{proposition}
Let $(M, g)$ be a manifold with asymptotic end $\mathbb{R}^{k}\times X$ and its natural compactification $M_1$ is spin. If either $X$ is enlargeable or $\hat{A}(X)\not=0$, then $M_{1}^{n}\#_{X}T^{k}\times X^{n-k}$ admits no metrics of positive scalar curvature.
\end{proposition}

\begin{proof}
If $X$ is enlargeable, so is $T^{k}\times X$. As there is a degree one map $M_{1}^{n}\#_{X}T^{k}\times X^{n-k} \longrightarrow T^{k}\times X^{n-k}$, there are no metrics of positive scalar curvature  on $M_{1}^{n}\#_{X}T^{k}\times X^{n-k}$ by Proposition 5.6 in \cite{6} if $M_{1}$ is, in addition, a spin manifold.

On the other hand, if $\hat{A}(X)$ is nonzero and $M_{1}$ is spin, then one can first construct a degree one map $f_{1}:M_{1}^{n}\#_{X}T^{k}\times X^{n-k}\to T^{k}\times X^{n-k}$ as follows. One collapses the part of $M_{1}^{n}$ in the generalized connected sum to the point of $T^k$ where we remove the $\epsilon$-ball, while mapping the compliment of $2\epsilon$-ball identically to $T^k$, and the annulus region between the $2\epsilon$-ball and $\epsilon$-ball smoothly to  the $2\epsilon$-ball. The boundary of the $\epsilon$-ball is mapped to the center point and the map on the $X$ factor is the identity. The map $f_1$ is followed by the projection map $f_{2}: T^{k}\times X^{n-k}\to T^{k}$ which has nonzero $\hat{A}$-degree by the assumption. Then we know that $M_{1}^{n}\#_{X}T^{k}\times X^{n-k}$ is enlargeable in dimension $k$ by Proposition 6.5 in \cite{6}.  Thus it is also without positive scalar curvature.  
\end{proof}

\section{The Green's function} \label{Green} 
In this section we estimate the asymptotic order of the Green's function $G_{g}$ of $\Delta_g$ when $r$ is sufficiency large. This will be used crucially in the next section to establish the asymptotic behavior of the solution of conformal Laplace equation. First of all, we have
\begin{proposition}\label{Greenexist} 
Let $(M, g)$ be asymptotic to $\mathbb R^k \times X$. Then $(M, g)$ is nonparabolic. That is, there exists a positive Green function $G_{g}$ on $M$.
\end{proposition} 
\begin{proof}
  By \cite[Coroallary 20.8]{LP}, which is attributed to Royden, nonparabolicity is a quasi-isometry invariant. Therefore we can assume that $g=g_{\mathbb R^k} + g_X$ outside a compact region. Now for $p=(x_0, y_0), x_0\in \mathbb R^k, y_0\in X,$ outside this compact region, let $r((x, y))=d_{\mathbb R^k}(x, 0)$ denote the Euclidean distance function from $x\in \mathbb R^k$ to $0\in \mathbb R^k$. Then $r^{2-k}$ is a harmonic function away from the compact region whose infinimum is achieved at infinity. Therefore by \cite[Theorem 17.3]{LP}, positive Green's function exists on $M$.
\end{proof}

Our main result in this section is the asymptotic behavior of $G_g$ as one approaches infinity. The following lemma gives a rough bound on this asymptotic behavior.

For each fixed $p\in M$ and any $(x,y)\in M\backslash K\simeq (\mathbb R^k\backslash B_1(0)) \times X$, let $r((x, y))=d_{\mathbb R^k}(x, 0)$ as above.  
\begin{lemma}
 For any $\epsilon>0$, there exist a larger compact set $K'\supset K$ containing $p$ and constant $C=C(\epsilon)>0$ such that, for $(x, y)$ outside $K'$, 
  \begin{equation}\label{Greenorder}
  G_{g}((x,y),p)\leq Cr^{2-k+\epsilon}((x,y)).  
\end{equation}
\end{lemma}  
\begin{proof} For any $\epsilon>0$, note that 
\begin{eqnarray*}
\Delta_{g}r^{2-k+\epsilon}((x,y))&=&\Delta_{\hat{g}}r^{2-k+\epsilon}+(\Delta_{g}-\Delta_{\hat{g}})r^{2-k+\epsilon}\\
&=&-\epsilon(k-2-\epsilon)r^{-k+\epsilon}+O(r^{-k-\tau+\epsilon}),       
\end{eqnarray*}
It follows that there exist a larger compact set $K'\supset K$ such that when $(x, y)$ is outside $K'$,  $$\Delta_{g}r^{2-k+\epsilon}((x,y))\leq0.$$
Then we can choose $C>0$ such that 
$$Cr^{2-k+\epsilon}|_{\partial K'}\geq G_{g}((x,y),p)|_{\partial K'}.$$
By the maximal principle, 
$$
  G_{g}((x,y),p)\leq Cr^{2-k+\epsilon}((x,y))
$$
outside $K'$.
\end{proof}

In what follows, for $x, x_0 \in \mathbb R^k$, we denote by $\|x-x_0\|$ their Euclidean distance.
\begin{proposition}\label{Greenprop}
   The Green's function $G_{g}$ has the expansion, when $\|x-x_{0}\|$ is large and $(x_{0},y_{0})\in M\backslash K$,
   $$G_{g}((x,y),(x_{0},y_{0}))=\frac{C}{\|x-x_{0}\|^{k-2}}+O(\frac{1}{\|x-x_{0}\|^{\min\{k-1,k-2+\tau\}}})$$
   for $C>0$. Moreover 
   $$\|x-x_{0}\|^{k-2}\cdot \|\nabla G_{g}((x,y),(x_{0},y_{0}))\|\to 0\quad \text{as}\quad\|x-x_{0}\| \to\infty.$$
\end{proposition}
Before we give the proof of the proposition, we recall some definitions; see \cite[p.930]{Minerbe1}. Let $K$ be compact set such that $M\backslash K$ is diffeomorphic to $(\mathbb{R}^{k}\backslash B_{R}(0))\times X$ and $r$ as before. Then for $\delta \in \mathbb R$, the weighted $L^2$-space is
$$L^{2}_{\delta}(\Omega)=\left\{u\in L_{loc}^2\left|\int_{\Omega\backslash K}u^2r^{-2\delta}d\operatorname{vol}_{\hat{g}}<\infty\right.\right\}$$
with the norm
$$\|u\|_{L_{\delta}^2(\Omega)}=\left(\int_{\Omega\cap K}u^2d\operatorname{vol}_{g}+\int_{\Omega\backslash K}u^2r^{-2\delta}d\operatorname{vol}_{\hat{g}}\right)^{\frac{1}{2}}.$$
Any $u\in L_{loc}^{2}(M\backslash K)$ can be written as $u=\Pi_{0}u+\Pi_{\perp}u$, where 
$$(\Pi_{0}u)(x)=\frac{1}{\operatorname{Vol}_{X}}\int_{X}u(x,y)d\operatorname{vol}_{y},\quad \Pi_{\perp}u=u-\Pi_{0}u.$$
Then, for any $\delta, \epsilon \in \mathbb R$, 
     $$L^{2}_{\delta,\epsilon}(\Omega)=\left\{u\in L_{loc}^2\left|\|\Pi_{0} u\|_{L_{\delta}^2(\Omega\backslash K)}<\infty\;\text{and}\; \|\Pi_{\perp} u\|_{L_{\epsilon}^2(\Omega\backslash K)}<\infty\right.\right\}$$
     with the norm      
     $$\|u\|_{L_{\delta,\epsilon}^2(\Omega)}=\left(\|u\|_{L_{\delta}^2(K\cap\Omega)}+\|\Pi_{0} u\|_{L_{\delta}^2(\Omega\backslash K)}+\|\Pi_{\perp} u\|_{L_{\epsilon}^2(\Omega\backslash K)}\right)^{\frac{1}{2}}.$$
     Finally we define the Sobolev space
     \begin{eqnarray*}
         H_{\delta}^2&=&\left\{u\in H_{loc}^{2}\left|\|\nabla^{g}d\Pi_{0} u\|_{L_{\delta-2}^2(K^{c})}+\|d\Pi_{0} u\|_{L_{\delta-1}^2(K^{c})}+\|\Pi_{0} u\|_{L_{\delta}^2(K^{c})}<\infty\right.\right.\\
         &&\quad\text{and}\;\left.\|\nabla^{g}d\Pi_{\perp} u\|_{L_{\delta-2}^2(K^{c})}+\|d\Pi_{\perp} u\|_{L_{\delta-2}^2(K^{c})}+\|\Pi_{\perp} u\|_{L_{\delta-2}^2(K^{c})}<\infty\right\}.
     \end{eqnarray*}

 Before we go into the proof of Proposition \ref{Greenprop}, we also need the following proposition from \cite{Minerbe1} which we quote here for convenience. The proposition gives asymptotic behavior of solutions to $\Delta_{g}u=f$ in terms of that of $f$ as well as those of harmonic functions on $\mathbb R^k$ expressed in the spectral decomposition of the Laplace operator $\Delta_{S}$ on the unit sphere $S^{k-1}$. 
 
 Recall that the eigenvalues of $\Delta_{S}$ are $\lambda_{j}=j(k-2+j), j\in \mathbb N$. Let $E_{j}$ be the eigenspace of $\Delta_{S}$ with eigenvalues $\lambda_{j}=j(k-2+j)$. Let $\delta_{j}=\frac{k}{2}+j, j\in \mathbb N$ and we call $\delta \in \mathbb R$ noncritical if $\delta\not=\delta_j$ and $\delta \not= 2-\delta_j$, for any $j\in \mathbb N$.
 
   \begin{proposition}[Proposition 4 in \cite{Minerbe1}]\label{gradientproposition}
Suppose $\Delta_{g}u=f$ with $u$ in $L_{\delta}^{2}(K^{c})$ and $f$ in $L_{\delta'-2}^{2}(K^{c})$ for $K$ compact and noncritical exponents $\delta>\delta'$. Then, up to enlarging $K$, there is an element $v$ of $L^{2}_{\delta',\delta'-2}(K^{c})$ such that $u-v$ is a linear combination of the following function:
\begin{enumerate}
  \item $\aleph^{+}_{j,\phi_{j}}$ with $\phi_{j}$ in $E_{j}$, if $\delta'<\delta_{j}<\delta$;
  \item $\aleph^{-}_{j,\phi_{j}}$ with $\phi_{j}$ in $E_{j}$, if $\delta'<2-\delta_{j}<\delta$.
\end{enumerate}
where $\aleph^{\pm}_{j,\phi_{j}}=\Phi (r^{\nu_{j}^{\pm}}\phi_{j}+v_{j}^{\pm})$ with $\nu_{j}^{+}=j$, $\nu_{j}^{-}=2-k-j$.  $\Phi$ is a smooth cut-off function which is $0$ on a compact set and $1$ on $K$. Moreover, $\Delta_{g} v^{\pm}_{j}=-\Delta_{g}(r^{\nu_{j}^{\pm}}\phi_{j})$.
   \end{proposition}
   
 \begin{remark}
   In \cite{Minerbe1}, the author deals explicitly with the case when $X=S^{1}$. But the method and result in \cite{Minerbe1} generalize to the general $X$ without change.
   \end{remark}
   
\begin{proof}[Proof of Proposition \ref{Greenprop}]
   For fixed $y\in M$, consider the equation,  
   \begin{eqnarray}\label{equation18}
     \Delta_{g}[\chi(x)G_{g}] &=&2\langle\nabla_{g} G_{g},\nabla_{g}\chi(x)\rangle +G_{g}\Delta_{g}\chi(x)
   \end{eqnarray}
   where $\chi(x)$ is a cut-off function which is $0$ in $B_{y}(r_{1})$ and $1$ in $B^{c}_{y}(r_{2})$ for $r_{2}>>r_{1}$.
    The order of the right hand term of equation (\ref{equation18}) is arbitrary, hence it lies in $L_{\delta'-2}^2(B_{R_{0}}^{c})$ especially for any $\delta'>-\frac{k}{2}$. Also $\chi(x)G_{g}\in L_{\delta}^2(B_{R_{0}}^{c})$ for any $\delta>-\frac{k}{2}+2+\epsilon$ by \eqref{Greenorder}. Thus we can choose $\delta,\delta'$ such that only $\delta'<2-\delta_{j}<\delta$ is possible for some $j$. Thus, by Proposition \ref{gradientproposition}, we have, 
    $$\chi(x)G_{g}=v+\sum_{j}\aleph^{-}_{j,\phi_{j}}.$$
    where $v\in L^2_{\delta',\delta'-2}(B_{R_{0}^{c}})$, the sum runs over $j\in \mathbb N$ such that $\delta'< 2-\delta_{j}<\delta$, and 
    $$\Delta_{g}v=2\langle\nabla_{g} G_{g},\nabla_{g}\chi(x)\rangle +G_{g}\Delta_{g}\chi(x).$$
     Using Moser iteration as in the proof of Lemma 6 in \cite{Minerbe1}, we have $v=O(r^{-\frac{k}{2}+\delta'})$ for (possibly different but still arbitrary) $\delta'>-\frac{k}{2}$.
     Since $\delta>-\frac{k}{2}+2+\epsilon$ and $\delta'>-\frac{k}{2}$ are arbitrary. we choose $\varepsilon'>0$ such that $\delta'=-\frac{k}{2}+\varepsilon'$ and $\delta=-\frac{k}{2}+2+\varepsilon+\varepsilon'$ and $1>\varepsilon+\varepsilon'>0$. Then the only $j\in \mathbb{N}$ satisfying 
     $$-\frac{k}{2}+\varepsilon'<2-\frac{k}{2}-j<-\frac{k}{2}+2+\varepsilon+\varepsilon'.$$
     are $j=0, 1$. Thus $\aleph^{-}_{j,\phi_{j}}=\Phi (r^{2-k-j}\phi_{j}+v_{j}^{-})$, where 
     On the other hand, $\Delta_{g} v^{-}_{j}=-\Delta_{g}(r^{\nu_{j}^{-}}\phi_{j})$, $\Delta_{g}(r^{\nu_{j}^{-}}\phi_{j})\in L^2_{-\eta-\tau}$ with $\eta>\delta_{j}$  and thus we can find $v^{-}_{j}\in H^2_{2-\eta-\tau}$ again as in  the proof of Lemma 6 in \cite{Minerbe1}.  Therefore we have $v_{j}^{-}=O(r^{2-k-\tau})$.
     Hence, 
     \begin{eqnarray}
       \chi(x)G_{g} &=& O(\frac{1}{r^{\min{\{k-1,k+\tau-2\}}}})+\frac{C\Phi}{r^{k-2}}.
     \end{eqnarray}
     Since we can differentiate the both sides of (\ref{equation18}) and use Proposition \ref{gradientproposition} again and repeat the above process, we also get $\|\nabla G_{g}\|=O(\frac{1}{r^{k-1}})$ for $r>>1$. By the positivity of $G_{g}$,
     \begin{eqnarray*}
       G_{g} &=& \frac{C}{r^{k-2}}+O(\frac{1}{r^{\min\{k-1,k-2+\tau\}}})
     \end{eqnarray*}
     for $C>0$ and $r>>1$.
\end{proof}

\section{The compactification} \label{compact}

Let $(M, g)$ be a manifold asymptotic to $\mathbb R^k \times X$ whose scalar curvature is nonnegative but its mass is negative. 
Following the general strategy of \cite{1}, we compactify $(M,g)$ by cutting $M$ off a large compact set $K$ such that $\partial K=(\partial[0,1]^{k})\times X$ and gluing the opposite faces of $[0,1]^{k}$, with the resulting manifold $(M_{1}^{n}\#_{X}T^{k}\times X,g)$. In order for $(M_{1}^{n}\#_{X}T^{k}\times X,g)$ to still have a metric  with positive scalar curvature we deform $g$ so that it is the product of Euclidean metric with $g_X$ outside a compact set while still maintains nonnegative scalar curvature. 

There are two key steps.
\begin{itemize}
  \item Step 1 If $(M,g)$ is asymptotically $(\mathbb R^k \times X,g_{\mathbb{R}^{k}}+g_{X})$ and $Sc_g\geq 0$ but $m(g)<0$, then there is a  metric      $\tilde{g}=\tilde{u}^{\frac{4}{n-2}}(g_{\mathbb{R}^{k}}+g_{X})$ with $\tilde{u}=1+\frac{\tilde{m}}{r^{k-2}}+O(r^{1-k})(\tilde{m}<0)$ and $Sc_{\tilde{g}}\geq 0$, $Sc_{\tilde{g}}=0$  outside a large compact set.
  \item step 2 This is an observation due to J.Lohkamp. If $(M,g)$ with $g=u^{\frac{4}{n-2}}(g_{\mathbb{R}^{K}}+g_{X})$ and $u=1+\frac{m}{r^{k-2}}+O(r^{1-k})(m<0)$ and $Sc_g \geq 0$, $Sc_{g}=0$  outside a large compact set, then there exists a metric $\tilde{g}$ with $Sc(\tilde{g})\geq 0$ and $\tilde{g}=g_{\mathbb{R}^k}+g_{X}$ near $\infty$.
\end{itemize}
We first prove the Sobolev inequality on $(M,g)$ which is used in the rest part of the paper (as well as in the previous part where Moser iteration is involved). Recall that $M\backslash K \simeq (\mathbb R^k - B_R(0))\times X$. For $r>R$, let 
$B_{r}=K\cup (B_{r}(0)\backslash B_{R}(0))\times X$. 
\begin{lemma}
  There is a Sobolev constant $c>0$ not depending on $B_{r}$ such that, for
  $f\in C^{\infty}_{0}(B_{r})$, 
  \begin{eqnarray}
  \left(\int_{B_{r}}f^{\frac{2n}{n-2}}d\operatorname{vol}_{g}\right)^{\frac{n-2}{2n}}\leq c\left(\int_{B_{r}}|\nabla f|^2d\operatorname{vol}_{g}\right)^{\frac{1}{2}}. 
  \end{eqnarray}
\end{lemma}
\begin{proof}
   Write $f=\Pi_{0} f+\Pi_{\perp} f$ in $B_{r}\backslash K$ where $\Delta_{g_{X}}\Pi_{0} f=0$, $\Pi_{0} f=\frac{1}{\operatorname{Vol(X)}}\int_{X}fd\operatorname{vol}_{g_{X}}$ and $\Pi_{\perp}f=f-\Pi_{0} f$. Then
   \begin{eqnarray*}
\nabla_{g_{\mathbb{R}^{k}}}\Pi_{0}f&=&\nabla_{g_{\mathbb{R}^{k}}}\left(\frac{1}{\operatorname{Vol(X)}}\int_{X}fd\operatorname{vol}_{g_{X}}\right)\\
  &=&\frac{1}{\operatorname{Vol(X)}}\int_{X}\left(\nabla_{g_{\mathbb{R}^{k}}}f\right)d\operatorname{vol}_{g_{X}}\\
\nabla_{g_{X}}\Pi_{0}f&=&0,
   \end{eqnarray*}
   and
   \begin{eqnarray*}
  |\nabla_{g_{\mathbb{R}^{k}}}\Pi_{0}f|^2
  &=&\left|\frac{1}{\operatorname{Vol(X)}}\int_{X}\left(\nabla_{g_{\mathbb{R}^{k}}}f\right)d\operatorname{vol}_{g_{X}}\right|^2\\
  &\leq&\frac{1}{\operatorname{Vol(X)}}\int_{X}\left|\nabla_{g_{\mathbb{R}^{k}}}f\right|^2d\operatorname{vol}_{g_{X}}
   \end{eqnarray*}
   We only need to consider the metric $\hat{g}$ since it is equivalent to $g$ in $B_{r}\backslash K$. 
  Then, for some $c_{1}, c_{2}>0$
  \begin{eqnarray*}
  \left(\int_{B_{r}\backslash K}(\Pi_{0} f)^{\frac{2n}{n-2}}d\operatorname{vol}_{\hat{g}}\right)^{\frac{n-2}{2n}}&\leq& c_{1}\left(\int_{B_{r}\backslash K}|\nabla_{g_{\mathbb{R}^{k}}} \Pi_{0} f|^2d\operatorname{vol}_{\hat{g}}\right)^{\frac{1}{2}}\\\nonumber
  &\leq& c_{1}\left(\int_{B_{r}\backslash K}\left[\frac{1}{\operatorname{Vol(X)}}\int_{X}\left|\nabla_{g_{\mathbb{R}^{k}}}f\right|^2d\operatorname{vol}_{g_{X}}\right]d\operatorname{vol}_{\hat{g}}\right)^{\frac{1}{2}}\\
  &\leq&c_1\left(\int_{B_{r}\backslash K}\left|\nabla_{g_{\mathbb{R}^{k}}}f\right|^2d\operatorname{vol}_{\hat{g}}\right)^{\frac{1}{2}}\\
  &\leq&c_1\left(\int_{B_{r}\backslash K}\left|\nabla_{\hat{g}}f\right|^2d\operatorname{vol}_{\hat{g}}\right)^{\frac{1}{2}}
  \end{eqnarray*}
  \begin{eqnarray*}
  \left(\int_{B_{r}\backslash K}(\Pi_{\perp} f)^{\frac{2n}{n-2}}d\operatorname{vol}_{\hat{g}}\right)^{\frac{n-2}{2n}}&\leq& c_{2}\left(\int_{B_{r}\backslash K}|\nabla_{g_{X}} \Pi_{\perp} f|^2d\operatorname{vol}_{\hat{g}}\right)^{\frac{1}{2}}\\\nonumber
  &=& c_{2}\left(\int_{B_{r}\backslash K}|\nabla_{g_{X}}f|^2d\operatorname{vol}_{\hat{g}}\right)^{\frac{1}{2}}\\
  &\leq&c_2\left(\int_{B_{r}\backslash K}|\nabla_{\hat{g}}f|^2d\operatorname{vol}_{\hat{g}}\right)^{\frac{1}{2}}
  \end{eqnarray*}
  Let $\chi$ be a cut-off function which  is $1$ on $K$ and is $0$  outside a larger compact set $K'$, such that $K\subset K'=B_{r_{1}} \subset B_{r}$. Then
  \begin{eqnarray*}
  \left(\int_{K} f^{\frac{2n}{n-2}}d\operatorname{vol}_{g}\right)^{\frac{n-2}{2n}}&\leq& \left(\int_{K'} (\chi f)^{\frac{2n}{n-2}}d\operatorname{vol}_{g}\right)^{\frac{n-2}{2n}}\\
  &\leq& C_{K'}\left(\int_{K'} \|\nabla(\chi f)\|^{2}d\operatorname{vol}_{g}\right)^{\frac{1}{2}}\\
  &\leq& C_{K'}\left(\int_{K'} \|\nabla f)\|^{2}d\operatorname{vol}_{g}\right)^{\frac{1}{2}}+ C_{K'}\left(\int_{K'} \|f\nabla\chi \|^{2}d\operatorname{vol}_{g}\right)^{\frac{1}{2}}\\
  &\leq& C_{K'}\left(\int_{K'} \|\nabla f)\|^{2}d\operatorname{vol}_{g}\right)^{\frac{1}{2}}\\
  &&+ C_{K'}\left(\int_{K'\backslash K} \|f\|^{\frac{2n}{n-2}}d\operatorname{vol}_{\hat{g}}\right)^{\frac{n-2}{2n}}\left(\int_{K'\backslash K} \|\nabla\chi \|^{n}d\operatorname{vol}_{\hat{g}}\right)^{\frac{1}{n}}.
  \end{eqnarray*}
  Using the first part we get the result.
\end{proof}
\begin{proposition}\label{proposition1}
  Suppose $(M^{n},g)$ is an asymptotically flat manifold with asymptotic end $\mathbb{R}^{k}\times X$. Then there exists a constant $\epsilon_{0}=\epsilon_{0}(g)$, such that if $f$ is a smooth function with compact support and $\|f_{-}\|_{L^{\frac{n}{2}}}<\epsilon_{0}$, then the equation
   \begin{equation}\label{EQu}
    \begin{cases}
      \Delta_{g} u-fu=0 \;\text{on}\; M \\
      u\to 1 \;\text{as}\; r\to \infty
    \end{cases}
   \end{equation}
has a unique positive solution. Moreover, near infinity $u$ has asymptotics
\begin{eqnarray*}
  u &=& 1-\frac{A}{r^{k-2}}+O(r^{-k+1})
\end{eqnarray*}
where $A=C\int_{M}fu\, d\operatorname{vol}_{g}$ for some $C>0$.
\end{proposition}
\begin{proof}
  Let $v=1-u$. Then \eqref{EQu} becomes
 \begin{equation}\label{FNEQ}
    \begin{cases}
      \Delta_{g} v-fv=-f \;\text{on}\; M \\
      v\to 0 \;\text{as}\; r \to \infty
    \end{cases}
   \end{equation}
   On a compact subset $B_{r}$, consider the Dirichlet problem
   \begin{equation}\label{NPEQ}
    \begin{cases}
      \Delta_{g} v_{r}-fv_{r}=-f \;\text{in}\; B_{r} \\
      v_{r}=0 \;\text{on}\; \partial B_{r}
    \end{cases}
   \end{equation}
   By Fredholm alternative, if the homogeneous equation
   \begin{equation}\label{UE}
    \begin{cases}
      \Delta_{g} v_{r}-fv_{r}=0 \;\text{in}\; B_{r} \\
      v_{r}=0 \;\text{on}\; \partial B_{r}
    \end{cases}
   \end{equation}
   has only zero solution, 
   then equation (\ref{NPEQ}) has a unique solution. Suppose $\omega$ is a solution of equation (\ref{UE}). Multiplying $\omega$ to both sides of (\ref{UE}) and integrating by parts, by the Holder inequality with $p=\frac{n}{2},q=\frac{n}{n-2}$ and the Sobolev inequality with $p=2,p^{\ast}=\frac{2n}{n-2}$, we have
   \begin{eqnarray*}
     \int_{B_{r}}|\nabla \omega|^2 d\operatorname{vol}_{g}&=& -\int_{B_{r}}f\omega^2d\operatorname{vol}_{g}\leq \int_{B_{r}}f_{-}\omega^2d\operatorname{vol}_{g}\\
     &\leq&\left(\int_{B_{r}}f_{-}^{\frac{n}{2}}d\operatorname{vol}_{g}\right)^{\frac{2}{n}}\left(\int_{B_{r}}\omega^{\frac{2n}{n-2}}d\operatorname{vol}_{g}\right)^{\frac{n-2}{n}}  \\
      &\leq&c_{1} \left(\int_{B_{r}}f_{-}^{\frac{n}{2}}d\operatorname{vol}_{g}\right)^{\frac{2}{n}}\left(\int_{B_{r}}|\nabla \omega|^2d\operatorname{vol}_{g}\right)
   \end{eqnarray*}
   Thus if $\|f_{-}\|_{L^{\frac{n}{2}}}< \frac{1}{c_{1}}$, then $\omega=0$. Therefore, equation (\ref{NPEQ}) has a unique solution $v_{r}$.
   Multiplying $v_{r}$ to both sides of (\ref{NPEQ}), using Holder inequality and Sobolev inequality again,
   \begin{eqnarray*}
     \int_{B_{r}}|\nabla v_{r}|^2d\operatorname{vol}_{g} &\leq& \int_{B_{r}}f_{-}v^2_{r}d\operatorname{vol}_{g}+\int_{B_{r}}fv_{r}d\operatorname{vol}_{g} \\
      &\leq& c_{1}\left(\int_{B_{r}}f_{-}^{\frac{n}{2}}d\operatorname{vol}_{g}\right)^{\frac{2}{n}}\left(\int_{B_{r}}|\nabla v_r|^2d\operatorname{vol}_{g}\right)\\
      &&+\left(\int_{B_{r}}f^{\frac{2n}{n+2}}d\operatorname{vol}_{g}\right)^{\frac{n+2}{2n}}\left(\int_{B_{r}}v_{r}^{\frac{2n}{n-2}}d\operatorname{vol}_{g}\right)^{\frac{n-2}{2n}}\\
      &\leq&c_{1}\left(\int_{B_{r}}f_{-}^{\frac{n}{2}}d\operatorname{vol}_{g}\right)^{\frac{2}{n}}\left(\int_{B_{r}}|\nabla v_r|^2d\operatorname{vol}_{g}\right)\\
      &&+c_{1}\left(\int_{B_{r}}f^{\frac{2n}{n+2}}d\operatorname{vol}_{g}\right)^{\frac{n+2}{2n}}\left(\int_{B_{r}}| \nabla v_r|^{2}d\operatorname{vol}_{g}\right)^{\frac{1}{2}}
   \end{eqnarray*}
   Then there is a constant $c_{2}$ depending on $(M,g,f)$ such that $\|v_{r}\|_{L^{\frac{2n}{n-2}}}<c_{2}$ and $\|\nabla v_{r}\|_{L^2}<c_{2}$. The standard theory of elliptic equations concludes that $v_{r}$ has uniformly bounded $C^{2,\alpha}$ norm. By Arzela-Ascoli we may pass to a limit and conclude that equation (\ref{FNEQ}) has a solution.

   A similar argument proves that the solution of equation (\ref{EQu}) is nonnegative everywhere. Otherwise there exists an open set $\Omega$ such that
   \begin{equation*}
    \begin{cases}
      \Delta_{g} u-fu=0 \;\text{in}\; \Omega \\
      u= 0 \;\text{on}\; \partial \Omega.
    \end{cases}
   \end{equation*}
   This contradicts with the Sobolev inequality and the choice of $\epsilon_{0}$ as above since $u$ is the nonzero solution. By the strong maximum principle $u$ is positive everywhere.

In the following, we use the asymptotic estimate of the Green's function $G$ and $\nabla G$ to obtain the asymptitc behavior of the solution. 
 Let $\phi(r)$ be a smooth cut-off function on $M$ with $\phi(r)=1$ on $M\backslash B_{R_1}$, $\phi(r)=0$ on $B_{R_2}$ for $R<R_{2}<R_{1}$. Let $dS$ be the boundary area form. Fix $(x, y)\in M\backslash B_{R_1}$ and choose $s$ sufficiently large so that $B_s$ contains $(x, y)$. 
Multiply $G=G(z, (x, y))$ to $\Delta_{g} v=fv-f$ and integrate on $B_{s}$ about the variable $z\in B_{s}$.
Then
\begin{eqnarray*}
  \int_{B_{s}}\phi G (fv-f)d\operatorname{vol}_{g} &=& \int_{B_{s}}\phi G \Delta_{g} vd\operatorname{vol}_{g}\\
  &=& -\int_{B_s}\langle\nabla(\phi G),\nabla v\rangle d\operatorname{vol}_{g}+\int_{\partial B_s}(\phi G)\langle\nu, \nabla v\rangle dS\\
  &=&\int_{B_s}G\Delta_{g}(\phi)  vd\operatorname{vol}_{g}+\int_{B_{s}}\phi v\Delta_{g} G  d\operatorname{vol}_{g}+2\int_{B_{s}}\langle\nabla\phi, \nabla G\rangle  vd\operatorname{vol}_{g}\\
  &&+\int_{\partial B_s}(\phi G)\langle\nu, \nabla v\rangle dS-\int_{\partial B_s}v\langle\nu, \nabla (\phi G)\rangle dS
\end{eqnarray*}
Since
\begin{eqnarray*}
  \left|\int_{\partial B_s}(\phi G)\langle\nu, \nabla v\rangle dS\right| &\leq& G((x,y),z)\int_{\partial B_s}|\langle\nu, \nabla v\rangle| dS\\
  &\leq& \max_{z\in \partial B_s}G((x,y), z)\int_{\partial B_s}|\nabla v|^2 dS.
\end{eqnarray*}
As $s\rightarrow \infty$, $\max_{z\in \partial B_s}G((x,y), z)\to 0$. Since $\int_{M}|\nabla v|^2 d\operatorname{vol}_{g}<\infty$, we deduce that 
$$\int_{\partial B_s}(\phi G)\langle\nu, \nabla v\rangle dS\to 0 , \mbox{as} \ s\to \infty.$$
Similarly we have
$$\int_{\partial B_s}v\langle\nu, \nabla (\phi G)\rangle dS \to 0 , \mbox{as} \ s\to \infty.$$
Thus, taking $s\rightarrow \infty$ so that $B_{s}\to M$, we obtain 
\begin{eqnarray*}
  \int_{M}\phi G (fv-f)d\operatorname{vol}_{g} &=& v((x,y))+\int_{M}G\Delta_{g}(\phi)  vd\operatorname{vol}_{g}+2\int_{M}\langle\nabla\phi, \nabla G\rangle  vd\operatorname{vol}_{g}.
\end{eqnarray*}
Therefore
\begin{eqnarray*}
  \lim_{r\to\infty}r^{k-2}v((x,y)) &=& -2\int_{M}\lim_{r\to\infty}r^{k-2}\langle\nabla\phi, \nabla G\rangle  vd\operatorname{vol}_{g}+\int_{M}\lim_{r\to\infty}r^{k-2}G\phi  (fv-f)d\operatorname{vol}_{g}\\
  &&+\int_{M}\lim_{r\to\infty}r^{k-2}G \Delta_{g}(1-\phi)  vd\operatorname{vol}_{g}\\
  &=&C\left[\int_{M}\phi  (fv-f)d\operatorname{vol}_{g}+\int_{M}\Delta_{g}(1-\phi)  vd\operatorname{vol}_{g}\right]\\
  &=&C\left[\int_{M}\phi  (fv-f)d\operatorname{vol}_{g}+\int_{M}(1-\phi) \Delta_{g} vd\operatorname{vol}_{g}\right]\\
  &=&C\int_{M}(fv-f)d\operatorname{vol}_{g}
\end{eqnarray*}
for $C>0$.
 \end{proof}
 
The next lemma is standard, which relates the masses of conformally related metrics.
   \begin{lemma}\label{lemma m}
     For metric $\tilde{g}=u^{\frac{4}{n-2}}((x,y))g$ outside a large compact set on $M^{n}$, and $u=1+\frac{m_0}{r^{k-2}}+O(r^{-k+1})$,
     $$m(\tilde{g})=m(g)+\frac{4(n-1)(k-2)}{n-2} m_0$$
   \end{lemma}
\begin{proof} Let $dS_{\rho}$ be the area form for the sphere of radius $\rho$ in $\mathbb{R}^{k}$.
  \begin{eqnarray*}
  m(\tilde{g}) &=&\frac{1}{\omega_{k}\operatorname{Vol}(X)}\lim_{\rho\to\infty}\int_{S_{\rho}\times X}\sum_{i,j}(\partial_{i}\tilde{g}_{ij}-\partial_{j}\tilde{g}_{aa})\frac{x^{j}}{\rho}dS_{\rho}d\operatorname{vol}_{g_{X}}  \\
               &=&\frac{1}{\omega_{k}\operatorname{Vol}(X)}\lim_{\rho\to\infty}\int_{S_{\rho}\times X}\sum_{i,j}(\partial_{i}(u^{\frac{4}{n-2}}g_{ij})-\partial_{j}(u^{\frac{4}{n-2}}g_{aa}))\frac{x^{j}}{\rho}dS_{\rho}d\operatorname{vol}_{g_{X}}\\
               &=&\frac{1}{\omega_{k}\operatorname{Vol}(X)}\lim_{\rho\to\infty}\int_{S_{\rho}\times X}\sum_{i,j}u^{\frac{4}{n-2}}(\partial_{i}g_{ij}-\partial_{j}g_{aa})\frac{x^{j}}{\rho}dS_{\rho}d\operatorname{vol}_{g_{X}}\\
               &&+\frac{1}{\omega_{k}\operatorname{Vol}(X)}\lim_{\rho\to\infty}\int_{S_{\rho}\times X}\sum_{i,j}\frac{4}{n-2}u^{\frac{4}{n-2}-1}(\partial_{i}ug_{ij}-\partial_{j}ug_{aa})\frac{x^{j}}{\rho}dS_{\rho}d\operatorname{vol}_{g_{X}}\\
               &=&\frac{1}{\omega_{k}\operatorname{Vol}(X)}\lim_{\rho\to\infty}\int_{S_{\rho}\times X}(1+\frac{m_0}{2\rho^{k-2}}+O(\rho^{-k+1}))^{\frac{4}{n-2}}\sum_{i,j}(\partial_{i}g_{ij}-\partial_{j}g_{aa})\frac{x^{j}}{\rho}dS_{\rho}d\operatorname{vol}_{g_{X}}\\
               &&+\frac{1}{\omega_{k}\operatorname{Vol}(X)}\lim_{\rho\to\infty}\int_{S_{\rho}\times X}          \sum_{i,j}\frac{4}{n-2}u^{\frac{4}{n-2}-1}(((2-k)m_0\rho^{1-k}\frac{x^{i}}{\rho}+O(\rho^{-k}))g_{ij}\\
              && -((2-k)m_0\rho^{1-k}\frac{x^{j}}{\rho}+O(\rho^{-k}))g_{aa})\frac{x^{j}}{\rho}dS_{\rho}d\operatorname{vol}_{g_{X}}\\
              &=&m(g)+\frac{4(n-1)(k-2)}{n-2}m_0 .
\end{eqnarray*}
\end{proof}

We are now ready to start the proof of compactification result.

\begin{proof}[Proof of Step 1]
     Following the proof of Proposition 4.11 in \cite{CLSZ} and Proposition 3.2 in \cite{ZJT}, write the metric $g$ as
   \begin{eqnarray*}
     g &=& (1+\frac{m_{1}}{r^{k-2}})^{\frac{4}{n-2}}\hat{g}+\bar{g}
   \end{eqnarray*}
outside a large compact set with $m_{1}=\frac{n-2}{4(n-1)(k-2)}m$ and
   \begin{eqnarray}\label{eqnmasszero}
     \lim_{\rho\to\infty}\int_{S_{\rho}\times X}\sum_{i,j}(\partial_{i}\bar{g}_{ij}-\partial_{j}\bar{g}_{aa})\frac{x^{j}}{\rho}dS_{\rho}d\operatorname{vol}_{g_{X}}&=&0
   \end{eqnarray}
     Let $\phi(r)$ be a cut-off function, $\phi(r)=1$ for $r\leq 2$ and $\phi(r)=0$ for $r\geq 3$ and $0\leq\phi(r)\leq 1$.
Define the metric $g^{\sigma}=(1+\frac{m_{1}}{r^{k-2}})^{\frac{4}{n-2}}\hat{g}+\phi(\frac{r}{\sigma})\bar{g}$.
Then $Sc_{g^{\sigma}}$ is a smooth function with compact support in $B_{3\sigma}\times X$. More precisely,
\begin{eqnarray*}
  Sc_{g^{\sigma}} &=&\begin{cases}
                       Sc_{g}, & \mbox{if } r\leq 2\sigma \\
                       O(\sigma^{-\tau-2}), & \mbox{if }2\sigma\leq r\leq 3\sigma \\
                       0, & \mbox{otherwise}.
                     \end{cases}
\end{eqnarray*}  Solve
\begin{eqnarray}\label{zeroscalarcurvature}
  \Delta_{g^{\sigma}}u-\frac{n-2}{4(n-1)}\varphi Sc_{g^{\sigma}}u &=& 0\\ \nonumber
  u&\to& 1\quad \text{as} \quad x\to\infty
\end{eqnarray}
where $\varphi(r)=\phi(\frac{r}{\sigma})=1$ for $2\sigma\leq r\leq 3\sigma$ and $\varphi(r)=0$ for $0\leq r\leq \sigma$ and $4\sigma\leq r<\infty$, $0\leq \varphi(r)\leq 1$.
Choose $\sigma$ sufficiency large to make $(\int_{M}|(\varphi Sc_{g^{\sigma}})_{-}|^{n/2}d\operatorname{vol}_{g^{\sigma}})^{2/n}=O(\sigma^{-\tau-2+2k/n})$ small.
By Proposition \ref{proposition1}, 
$$u=1-\frac{A_{\sigma}}{r^{k-2}}+O(r^{1-k})$$
and
$$A_{\sigma}=C\int_{M}\varphi Sc_{g^{\sigma}}ud\operatorname{vol}_{g^{\sigma}}.$$
Let $\tilde{g}=u^{\frac{4}{n-2}}g^{\sigma}$. Then 
\begin{eqnarray*}
  Sc_{\tilde{g}} &=& \frac{4(n-1)}{n-2}u^{-\frac{n+2}{n-2}}(-\Delta_{g^{\sigma}} u+\frac{n-2}{4(n-1)}Sc_{g^{\sigma}}u) \\
   &=& \frac{4(n-1)}{n-2}u^{-\frac{n+2}{n-2}}(-\Delta_{g^{\sigma}} u+\frac{n-2}{4(n-1)}\varphi Sc_{g^{\sigma}}u+\frac{n-2}{4(n-1)}(1-\varphi) Sc_{g^{\sigma}}u)\\
   &\geq&\frac{4(n-1)}{n-2}u^{-\frac{n+2}{n-2}}\frac{n-2}{4(n-1)}(1-\varphi) Sc_{g^{\sigma}}u\geq 0.
\end{eqnarray*}
Moreover,
$$m(\tilde{g})=-\frac{4(n-1)(k-2)}{n-2}A_{\sigma}+m(g^{\sigma})=-\frac{4(n-1)(k-2)}{n-2}A_{\sigma}+m$$ 
and thus 
\begin{eqnarray}
  |m(\tilde{g})-m(g)|=\frac{4(n-1)(k-2)}{n-2}|A_{\sigma}|.
\end{eqnarray}

  Thus if $|A_{\sigma}|<\epsilon$ can be made arbitrarily small for sufficiency large $\sigma$, then $|m(\tilde{g})-m(g)|<\epsilon$. Since $m(g)<0$, we can make $m(\tilde{g})<0$ by taking $\sigma$ sufficiency large.
  
  Now we prove $|A_{\sigma}|<\epsilon$ for sufficiency large $\sigma$. Let $v=1-u$, then 
  \begin{eqnarray*}
    |A_{\sigma}|&=&C\left|\int_{M}\varphi Sc_{g^{\sigma}}ud\operatorname{vol}_{g^{\sigma}}\right|\\
    &\leq&C\left|\int_{M}\varphi Sc_{g^{\sigma}}d\operatorname{vol}_{g^{\sigma}}\right|+C\left|\int_{M}\varphi Sc_{g^{\sigma}}vd\operatorname{vol}_{g^{\sigma}}\right|\\ 
    &\leq&C\left|\int_{M}\varphi Sc_{g^{\sigma}}d\operatorname{vol}_{g^{\sigma}}\right|+C\sigma^{-k+2}\int_{\{\sigma\leq r\leq 2\sigma\}\times X}\left| Sc_{g}\right|d\operatorname{vol}_{g}\\
    &&+C\left(\int_{\{2\sigma\leq r\leq 3\sigma\}\times X}\left|Sc_{g^{\sigma}}\right||v|d\operatorname{vol}_{g^{\sigma}}\right)\\
    &\leq&C\left|\int_{M}\varphi Sc_{g^{\sigma}}d\operatorname{vol}_{g^{\sigma}}\right|+O(\sigma^{-\tau})+C\sigma^{-k+2}\int_{\{\sigma\leq r\leq 2\sigma\}\times X}\left| Sc_{g}\right|d\operatorname{vol}_{g}\\
    &\leq&C\left|\int_{M}\varphi Sc_{g^{\sigma}}d\operatorname{vol}_{g^{\sigma}}\right|+O(\sigma^{-\tau})+C\sigma^{-k+2}
  \end{eqnarray*}
  by $Sc_{g}\in L^{1}(M)$.
  
  We now estimate the first term, where the integrand is nonzero only when $\sigma \leq r \leq 4\sigma$. In the asymptotic coordinates,  
  \begin{eqnarray*}
    Sc_{g^{\sigma}} &=& |g^{\sigma}|^{-\frac{1}{2}}\partial_{a}(|g^{\sigma}|^{\frac{1}{2}}(g^{\sigma})^{ab}(\Gamma_{b}-\frac{1}{2}\partial_{b}(\log|g^{\sigma}|))\\
    &&-\frac{1}{2}(g^{\sigma})^{ab}\Gamma_{a}\partial_{b}(\log|g^{\sigma}|)+(g^{\sigma})^{ab}(g^{\sigma})^{cd}(g^{\sigma})^{ef}\Gamma_{ace}\Gamma_{bdf},
  \end{eqnarray*}
  where $\Gamma_{abc}=\frac{1}{2}(g^{\sigma}_{bc,a}+g^{\sigma}_{ac,b}-g^{\sigma}_{ab,c})$ and $\Gamma_{c}=(g^{\sigma})^{ab}\Gamma_{abc}$. If
  \begin{eqnarray*}
    \left|\int_{\{\sigma\leq r\leq 4\sigma\}\times X}\varphi Sc_{g^{\sigma}}d\operatorname{vol}_{g^{\sigma}}\right|&=&  \int_{\{\sigma\leq r\leq 4\sigma\}\times X}\varphi Sc_{g^{\sigma}}d\operatorname{vol}_{g^{\sigma}},
  \end{eqnarray*}
  then
  \begin{eqnarray*}
    0&\leq&  \int_{\{\sigma\leq r\leq 2\sigma\}\times X}(1-\varphi) Sc_{g}d\operatorname{vol}_{g^{\sigma}}=  \int_{\{\sigma\leq r\leq 4\sigma\}\times X}(1-\varphi) Sc_{g^{\sigma}}d\operatorname{vol}_{g^{\sigma}}.
  \end{eqnarray*}
  If\begin{eqnarray*}
    \left|\int_{\{\sigma\leq r\leq 4\sigma\}\times X}\varphi Sc_{g^{\sigma}}d\operatorname{vol}_{g^{\sigma}}\right|&=&  -\int_{\{\sigma\leq r\leq 4\sigma\}\times X}\varphi Sc_{g^{\sigma}}d\operatorname{vol}_{g^{\sigma}},
  \end{eqnarray*}
  then
  \begin{eqnarray*}
    0&\geq&  \int_{\{\sigma\leq r\leq 4\sigma\}\times X}\varphi Sc_{g}d\operatorname{vol}_{g^{\sigma}}\\
    &=&  \int_{\{\sigma\leq r\leq 2\sigma\}\times X}(1-\varphi) Sc_{g^{\sigma}}d\operatorname{vol}_{g^{\sigma}}+\int_{\{2\sigma\leq r\leq 3\sigma\}\times X} Sc_{g^{\sigma}}d\operatorname{vol}_{g^{\sigma}}\\
    &=&  \int_{\{\sigma\leq r\leq 2\sigma\}\times X}(1-\varphi) Sc_{g}d\operatorname{vol}_{g^{\sigma}}+\int_{\{2\sigma\leq r\leq 3\sigma\}\times X} Sc_{g^{\sigma}}d\operatorname{vol}_{g^{\sigma}}\\
    &\geq&\int_{\{2\sigma\leq r\leq 3\sigma\}\times X} Sc_{g^{\sigma}}d\operatorname{vol}_{g^{\sigma}}.
  \end{eqnarray*}
  Therefore
  \begin{eqnarray*}
    \left|\int_{M}\varphi Sc_{g^{\sigma}}d\operatorname{vol}_{g^{\sigma}}\right| &\leq& \max\left\{\left|\int_{\{\sigma\leq r\leq 4\sigma\}\times X} Sc_{g^{\sigma}}d\operatorname{vol}_{g^{\sigma}}\right|,\left|\int_{\{2\sigma\leq r\leq 3\sigma\}\times X} Sc_{g^{\sigma}}d\operatorname{vol}_{g^{\sigma}}\right|\right\}.
  \end{eqnarray*}
  Thus, by $|g^{\sigma}|^{\frac{1}{2}}(g^{\sigma})^{ab}(\Gamma_{b}-\frac{1}{2}\partial_{b}(\log|g^{\sigma}|)=(g^{\sigma})_{ab,b}-(g^{\sigma})_{bb,a}+O(r^{-2\tau-1})$, there is
  \begin{eqnarray*}
    &&\int_{\{\sigma\leq r\leq 4\sigma\}\times X} Sc_{g^{\sigma}}d\operatorname{vol}_{g^{\sigma}}\\
     &=& \int_{\{ r=4\sigma\}\times X} |g^{\sigma}|^{\frac{1}{2}}(g^{\sigma})^{ab}(\Gamma_{b}-\frac{1}{2}\partial_{b}(\log|g^{\sigma}|)\nu^{a}dS_{4\sigma}d\operatorname{vol}_{g_{X}}\\
    &&-\int_{\{ r=\sigma\}\times X} |g^{\sigma}|^{\frac{1}{2}}(g^{\sigma})^{ab}(\Gamma_{b}-\frac{1}{2}\partial_{b}(\log|g^{\sigma}|)\nu^{a}dS_{\sigma}d\operatorname{vol}_{g_{X}}+O(\sigma^{-2\tau-2+k})\\
    &=& \int_{\{ r=4\sigma\}\times X} ((g^{\sigma})_{ab,b}-(g^{\sigma})_{bb,a})\nu^{a}dS_{4\sigma}d\operatorname{vol}_{g_{X}}\\
    &&-\int_{\{ r=\sigma\}\times X} ((g^{\sigma})_{ab,b}-(g^{\sigma})_{bb,a})\nu^{a}dS_{\sigma}d\operatorname{vol}_{g_{X}}+O(\sigma^{-2\tau-2+k})\\
    &=&  \int_{\{ r=4\sigma\}\times X} (g_{ab,b}-g_{bb,a})\nu^{a}dS_{4\sigma}d\operatorname{vol}_{g_{X}}\\
    &&-\int_{\{ r=\sigma\}\times X} (g_{ab,b}- g_{bb,a})\nu^{a}dS_{\sigma}d\operatorname{vol}_{g_{X}} \\
    && + \int_{S_{\sigma}\times X}\sum_{i,j}(\partial_{i}\bar{g}_{ij}-\partial_{j}\bar{g}_{aa})\nu^{a} dS_{\sigma}d\operatorname{vol}_{g_{X}}+O(\sigma^{-2\tau-2+k})
  \end{eqnarray*}
  where $\nu=(\nu^{a})$ is the outer normal vector. The difference of the first two terms is arbitrarily small if $\sigma$ is sufficiently large by the existence of mass \cite{B}. The third term is small for large $\sigma$ by (\ref{eqnmasszero}). The other term can be handled similarly.
   Therefore, we have $|A_{\sigma}|<\epsilon$ for sufficiency large $\sigma$.
  
  Thus,  $\tilde{g}=\tilde{u}^{\frac{4}{n-2}}(g_{\mathbb{R}^{K}}+g_{X})$ with $\tilde{u}=1+\frac{\tilde{m}}{r^{k-2}}+O(r^{1-k})(\tilde{m}<0)$ and $Sc_{\tilde{g}} \geq 0$. Moreover,  $Sc_{\tilde{g}}=0$ by (\ref{zeroscalarcurvature})  outside a large compact set $K$, for example $K=B_{3\sigma}$.
  \end{proof}
  
\begin{proof}[Proof of Step 2]
  From Step 1, we know that there is a metric $g=u^{\frac{4}{n-2}}(g_{\mathbb{R}^{k}}+g_{X})$ with $u=1+\frac{m}{r^{k-2}}+O(r^{1-k})(m<0)$ and $Sc_{g}=0$ outside a large compact set and $Sc_{g}\geq 0$ on $M$. Since
  \begin{eqnarray*}
  Sc_{g}&=&\frac{4(n-1)}{n-2}u^{-\frac{n+2}{n-2}}(-\Delta_{g_{\mathbb{R}^{k}}+g_{X}} u+\frac{n-2}{4(n-1)}Sc_{g_{\mathbb{R}^{k}}+g_{X}}u)\\
  &=&\frac{4(n-1)}{n-2}u^{-\frac{n+2}{n-2}}(-\Delta_{g_{\mathbb{R}^{k}}+g_{X}} u)
  \end{eqnarray*}
  therefore,  $\Delta_{g_{\mathbb{R}^{k}}+g_{X}} u=0$ outside a large compact set.
  By $u=1+\frac{m}{r^{k-2}}+O(r^{1-k})(m<0)$, we can take $s_{1}$ large enough such that $u<1$ on $\{s_{1}\}\times X$ and $s_{1}>3\sigma$. Let 
  $$\epsilon=1-\sup_{\{s_{1}\}\times X} u(x).$$
  Then $u>1-\frac{\epsilon}{4}$ in $r\geq s_{2}$ for sufficiently large $s_{2}>s_{1}$. Take a cutoff function $\zeta:[0,+\infty)\to [0,1-\frac{\epsilon}{2}]$ such that $\zeta(t)=t$ for $t\leq 1-\frac{3\epsilon}{4}$ and $\zeta(t)=1-\frac{\epsilon}{2}$ for $t\geq 1-\frac{\epsilon}{4}$ with $\zeta'\geq 0$ and $\zeta''\leq0$ as well as $\zeta''<0$ in $(1-\frac{3\epsilon}{4},1-\frac{\epsilon}{4})$. Let
  \begin{eqnarray*}
  v=\begin{cases}
    \zeta\circ u,& r\geq s_{1};\\
    u,& r\leq s_{1}.
  \end{cases}
  \end{eqnarray*}
  $v$ is a smooth function defined on entire $M$ and $v=u$ around $\{r=s_{1}\}\times X$. We also have
  $$\Delta_{g_{\mathbb{R}^{k}}+g_{X}}v=\zeta''|\nabla^{g_{\mathbb{R}^{k}}+g_{X}}u|^2+\zeta'\Delta_{g_{\mathbb{R}^{k}}+g_{X}}u\leq 0 \quad\text{in}\quad \{r\geq s_{1}\}\times X$$
  and $\Delta_{g_{\mathbb{R}^{k}}+g_{X}}v<0$ at some point in $\{s_{1}<s<s_{2}\}\times X$. Define
  $$\tilde{g}=\left(\frac{v}{u}\right)^{\frac{4}{n-2}}g.$$
  Then $Sc_{\tilde{g}}\geq 0$ with strict inequality at some point on $M$ and $\tilde{g}$ is the metric $g_{\mathbb{R}^{k}}+g_{X}$ near infinity.   

  Then we can cut $M$ off outside a large compact set $K$ such that $\partial K=(\partial[0,1]^{k})\times X$ and glue the opposite faces of $[0,1]^{k}$ , such that it becomes $(M_{1}^{n}\#_{X}T^{k}\times X,\tilde{g})$ with nonnegative scalar curvature and strictly positive scalar curvature at some point. It can then be deformed to a metric with positive scalar curvature.
  \end{proof}

\section{Rigidity} \label{rigidity}

This section is devoted to the proof of the rigidity part of the Positive Mass Theorem, Theorem \ref{maintheorem}. 
Thus let $(M, g)$ be a manifold asymptotic to $\mathbb R^k \times X$ whose scalar curvature is nonnegative but its mass is zero.  We first show that $Sc_{g}=0$, and then $\operatorname{ Ric}_{g}=0$.

  If the scalar curvature is not identically zero. Then $ Sc_{g}(p)>0$ for some $p\in M^{n}$. Choose compact sets $K_{1}, K_{2}$, $p\in K_{1}\subset K_{2}$, and let $\varphi$ be a nonnegative smooth function which is $1$ on $K_{1}$ and $0$ on $M\backslash K_{2}$. Solve the equation
  \begin{eqnarray*}
    \begin{cases}
      \Delta_{g} u-\frac{n-2}{4(n-1)}\varphi Sc_{g}u=0 \quad\text{on}\; M^{n}, \\
      u\to 1 \quad \text{as}\;r\to \infty.
    \end{cases}
  \end{eqnarray*}
  Then Proposition \ref{proposition1} gives a unique positive solution $u$. And the metric $\hat{g}=u^{\frac{4}{n-2}}g$ has  $Sc_{\hat{g}}\geq 0$ but
  $$m(\hat{g})=\frac{4(n-1)(k-2)}{n-2}m_{0}+m(g)=\frac{4(n-1)(k-2)}{n-2}m_{0},$$
  where $u=1+\frac{m_{0}}{r^{k-2}}+O(r^{1-k})$, and $m_{0}=-C\int_{M}\varphi Sc_{g}ud\operatorname{vol}_{g}<0$.
  That is,  $m(\hat{g})<0$, which is a contradiction.

  Next we prove that the Ricci curvature of $M^{n}$ is identically zero. Let $h$ be a compactly supported $(0,2)$ tensor and consider the deformation $g_{t}=g+th$. 
  For $t$ sufficiently small, since the scalar curvature depends smoothly on $t$, we have that $\|Sc(g_{t})\|_{L^{\frac{n}{2}}}$ will be small. Thus by Proposition \ref{proposition1} we can again solve the equation
    \begin{eqnarray*}
    \begin{cases}
      \Delta_{g_{t}} u_{t}-\frac{n-2}{4(n-1)}Sc_{g_{t}}u_{t}=0 \quad\text{on}\; M^{n} \\
      u_{t}\to 1 \quad \text{as}\;r\to \infty
    \end{cases}
  \end{eqnarray*}
   with a unique positive solution $u_{t}$. Define $\tilde{g}_{t}=u_{t}^{\frac{4}{n-2}}g_{t}$. Then $Sc_{\tilde{g}_{t}}=0$. Let $m(t)$ denote the mass of the metric $\tilde{g}_{t}$. Using the asymptotic formula again we see that
  $$m(t)=-C\int_{M}Sc_{g_{t}}u_{t}d\operatorname{vol_{g_{t}}},$$
  therefore $m(t)$ is $C^{1}$ differentiable about $t$. Taking its first derivative at $t=0$, and use the facts that $u_{0}\equiv 1$, $Sc_{g_{0}}=0$, we have
  \begin{eqnarray*}
    \left.\frac{d}{dt}\right|_{t=0}m(t) &=& -C\int_{M}\dot{Sc}(0)d\operatorname{vol}_{g} \\
     &=&-C\int_{M}(\nabla_{i}\nabla_{j} h_{ij}-\Delta_{g} (g^{ij}h_{ij})-\langle \operatorname{Ric_{g}},h\rangle)d\operatorname{vol}_{g}\\
     &=&C\int_{M}\langle \operatorname{Ric_{g}},h\rangle d\operatorname{vol}_{g}
  \end{eqnarray*}
  If $\operatorname{Ric}_{g}$ is not identically zero then taking $h=\eta \operatorname{Ric}_{g}$, with $\eta$ is a cutoff function yields that
  $$\left.\frac{d}{dt}\right|_{t=0}m(t)<0.$$
  This means that for some small $t$, $m(t)<0$, again a contradiction.
  Then $(M^{n},g)$ is Ricci flat. Since $\operatorname{ Ric}_{g}=\operatorname{ Ric}_{g_{\mathbb{R}^{k}}}+\operatorname{ Ric}_{g_{X}}+O(r^{-\tau})$ outside a compact set, $\operatorname{ Ric}_{g_{X}}\equiv 0$, i.e. $(X,g_{X})$ is $Ricci$ flat.

  As in the proof of lemma 6 in \cite{Minerbe1} and Proposition 4.12 in \cite{CLSZ}, we can assume that $\tau\leq k-2$. We first find smooth functions on $M$, $y^i, i=1, \cdots, k$, such that $\Delta_{g}y^{i}=0 $ and $y^1, \cdots, y^k$ form an asymptotic coordinate system for the $\mathbb R^k$ factor. Let $x^i \in C^{\infty}(M\backslash K), i=1, \cdots, k$, be an asymptotic coordinate system for the $\mathbb R^k$ factor. Let $\chi$ is a cut-off function vanishing on $B_{r_1}$ and is identically $1$ on $M\backslash B_{r_2}$ for some $r_1 < r_2$.  Since $\Delta_{g}(\chi x^{i})=O(r^{-\tau-1})\in L^{2}_{\delta-2}(M)$ for $\delta>1+\frac{k}{2}-\tau$, there exists $u^{i}\in H^{2}_{\delta}(M)$ such that $\Delta_{g}u^{i}=\Delta_{g}(\chi x^{i})$  by Corollary 2 in \cite{Minerbe1}. Set $y^{i}=\chi x^{i}-u^{i}$. Then $\Delta_{g}y^{i}=0$ for $1\leq i\leq k$ on $M$. By the Moser iteration and Schauder estimate, 
  $$|u^{i}|+r|\partial u^{i}|+r^2|\partial^2 u^{i}|=O(r^{1-\delta'}),\quad \delta'=1+\frac{k}{2}-\delta.$$
  Fix $\epsilon_{1}>0$ small enough such that if $\delta=1+\frac{k}{2}-\tau+\epsilon_{1}$, then $\delta'=\tau-\epsilon_{1}>\frac{k-2}{2}$.
  With a similar analysis on the derivative of $\Delta_{g}u^i=\Delta_{g}(\chi x^{i})$, one also deduces
  $$r^3|\partial^3 u^{i}|=O(r^{1-\delta'}).$$
  The above computation is in the coordinate system $\{x^1,\cdots,x^{k},\{f_{\alpha}\}\}$.

  Since $y^{i}=\chi x^{i}-u^{i}$,  $\{y^1,\cdots,y^{k},\{f_{\alpha}\}\}$ form a coordinate system outside a large compact set by the above estimates. With respect the new coordinate system $y^{i}=\chi x^{i}-u^{i}$ and $\{y^1,\cdots,y^{k},\{f_{\alpha}\}\}$ the metric $g$ can be written as $g=g'_{0}+\omega$ where $g_{0}'=dy^2+g_{X}$ and $\omega$ satisfies
  $$|\omega|+r|\partial \omega|+r^2|\partial^2 \omega|=O(r^{-\tau}),\quad \tau>\frac{k-2}{2}.$$
  Now we compute in this new coordinate system and the derivative will be taken with respect to $y^{i}=\chi x^{i}-u^{i}$ and $\{y^1,\cdots,y^{k},\{f_{\alpha}\}\}$. Let $\partial_{i}=\partial_{y^{i}}$, $g_{ij}=g(\frac{\partial}{\partial y^{i}},\frac{\partial}{\partial y^{j}})$ and $r=|y|$.

  From the Bochner formula and $\Delta_{g}y^{i}=0$, we have
  \begin{eqnarray*}
    \Delta_{g}\left(\frac{1}{2}|dy^{i}|^2_{g}\right) &=& |\nabla_{g}dy^{i}|^2_{g}, \\
    \Delta_{g}\left(\frac{1}{2}g^{ij}\right) &=& g(\nabla_{g}dy^{i},\nabla_{g}dy^{j}).
  \end{eqnarray*}
  Then $\Delta_{g}(g^{ij}-\delta_{ij})=O(r^{-2\delta'-2})$. From the above discussion we conclude that
  \begin{eqnarray*}
    g^{ij}-\delta_{ij} &\in& L^{2}_{\mu}\quad\text{for any}\quad \mu>\frac{k}{2}-\delta' \\
    \Delta_{g} (g^{ij}-\delta_{ij}) &\in& L^{2}_{\mu'-2}\quad\text{for any}\quad \mu'>\frac{k}{2}-2\delta'.
  \end{eqnarray*}
  Since $\delta'\leq \tau\leq k-2$ and $\delta'>\frac{k-2}{2}$, therefore we can choose $1-\frac{k}{2}<\mu'<2-\frac{k}{2}<\mu\leq \frac{k}{2}$. By Proposition \ref{gradientproposition},
  $$g^{ij}=\delta_{ij}-c_{ij}r^{2-k}+v^{ij}$$
  where $c_{ij}$ are constants with $c_{ij}=c_{ji}$ and $v^{ij}\in H^{2}_{\mu'}$.
  After repeating the argument for $u^{i}$, we see that $v^{ij}$ are higher-order error terms satisfying, for some small $\epsilon_2>0$,
  $$|v^{ij}|+r|\partial v^{ij}|+r^2|\partial^2 v^{ij}|=O(r^{2-k-\epsilon_{2}}).$$
  After a possible orthogonal transformation of $\{y^{1},\cdots,y^{k}\}$, we can assume $c_{ij}=c_{i}\delta_{ij}$ without loss of generality. Finally we have
  $$g_{ij}=\delta_{ij}+\omega_{ij},\quad g_{i\alpha}=\omega_{i\alpha},\quad g_{\alpha\alpha}=1+\omega_{\alpha\alpha},$$
  such that
  $$|\omega^{ij}|+r|\partial \omega^{ij}|+r^2|\partial^2 \omega^{ij}|=O(r^{2-k-\epsilon_{2}}),$$
  and
  $$|\omega^{\alpha\alpha}|+r|\partial \omega^{\alpha\alpha}|+r^2|\partial^2 \omega^{\alpha\alpha}|=O(r^{-\delta'}),\quad \delta'>\frac{k-2}{2}.$$
  $$|\omega^{i\alpha}|+r|\partial \omega^{i\alpha}|+r^2|\partial^2 \omega^{i\alpha}|=O(r^{-\delta'}),\quad \delta'>\frac{k-2}{2}.$$
  By the relation $y^{i}=\chi x^{i}- u^{i}$ and the estimate of $u^{i}$, we see that the mass $m(g)$ calculated with respect to  $\{y^1,\cdots,y^{k},\{f_{\alpha}\}\}$ is the same as that with respect to $\{x^1,\cdots,x^{k},\{f_{\alpha}\}\}$. Since $\Delta_{g}y^{i}=0$, a straightforward computation yields
  \begin{equation}\label{massequation1}
    (k-2) \left(c_{i}-\frac{1}{2}\left(\sum_{j=1}^{k}c_{j}\right)\right)\frac{y^{i}}{|y|^{k}}+\sum_{\alpha}\partial_{\alpha}g_{i\alpha}+\frac{1}{2}\sum_{\alpha}\partial_{i}g_{\alpha\alpha}+O(r^{-\mu''})=0
  \end{equation}
  where $\mu''=\min\{k-1+\epsilon_{2},2\delta'+1\}>k-1$. Integrating (\ref{massequation1}) and summing over $i$ gives
  \begin{eqnarray*}
    \lim_{\rho\to\infty}\int_{S_{\rho}\times X}(\partial_{i}g_{\alpha\alpha})\frac{y^{i}}{\rho}dS_{\rho}d\operatorname{vol}_{g_{X}} &=& \lim_{\rho\to\infty}\int_{S_{\rho}\times X}\frac{(k-2)^2}{k}\left(\sum_{i=1}^{k}c_{i}\right)\rho^{1-k} dS_{\rho}d\operatorname{vol}_{g_{X}}\\
    &=&\frac{(k-2)^2}{k}\omega_{k}\left(\sum_{i=1}^{k}c_{i}\right)\operatorname{Vol}(X).
  \end{eqnarray*}
  And a direct computation yields
  \begin{eqnarray*}
    \lim_{\rho\to+\infty}\int_{S_{\rho}\times X}(\partial_{j}g_{ij}-\partial_{i}g_{jj})\frac{y^{i}}{\rho}dS_{\rho}d\operatorname{vol}_{g_{X}} &=& \lim_{\rho\to+\infty}\int_{S_{\rho}\times X}\frac{(k-2)(k-1)}{k}\left(\sum_{i=1}^{k}c_{i}\right)\rho^{1-k}dS_{\rho}d\operatorname{vol}_{g_{X}}\\
    &=&\frac{(k-2)(k-1)}{k}\omega_{k}\left(\sum_{i=1}^{k}c_{i}\right)\operatorname{Vol}(X)
  \end{eqnarray*}
  Thus
  \begin{equation}\label{massresult}
    m(g)=\frac{k-2}{2}\left(\sum_{i=1}^{k}c_{i}\right).
  \end{equation}

  Since $\Delta_{g}dy_{i}=(dd^{*}+d^{*}d)dy_{i}=d\Delta_{g} y_{i}=0$, by the Weitzenbovk formula for $1$-form $\omega$, i.e.
  $$\Delta_{g} \omega=\nabla^{\ast}\nabla \omega+\operatorname{Ric}(\omega^{\#},\cdot).$$
  one has, 
  \begin{eqnarray*}
    0 &=& \sum_{i}\int_{M}\langle\nabla^{\ast}\nabla dy_{i},dy_{i}\rangle d\operatorname{vol}_{g} \\
     &=& \sum_{i}\int_{M}\langle\nabla dy_{i},\nabla dy_{i}\rangle d\operatorname{vol}_{g}- \sum_{i}\lim_{\rho\to \infty}\int_{S^{k-1}_{\rho}\times X}\langle\nabla_{a} dy_{i},dy_{i} \rangle\nu^{a} dS_{\rho}d\operatorname{vol}_{g_{X}}
  \end{eqnarray*}
  where $\nu=(\nu^{1},\cdots,\nu^{n})$ is the unit outer normal.
  On the other hand,
  \begin{eqnarray*}
    \nabla_{a}dy^{i} &=& -\Gamma_{ak}^{i}dy^{k}-\Gamma_{a\alpha}^{i}df_{\alpha}.
  \end{eqnarray*}
  Hence,
  \begin{eqnarray*}
   && \lim_{\rho\to \infty} \sum_{i}\int_{S^{k-1}_{\rho}\times X}\langle\nabla_{a} dy^{i},dy^{i} \rangle\nu^{a} dS_{\rho}d\operatorname{vol}_{g_{X}}\\
    &=&\lim_{\rho\to \infty} \int_{S^{k-1}_{\rho}\times X}(\partial_{j}g_{ij}-\partial_{i}g_{aa}+\frac{1}{2}\partial_{i}g_{\alpha\alpha})\frac{y^{i}}{\rho} dS_{\rho}d\operatorname{vol}_{g_{X}}\\
    &=&\frac{1}{2}(k-2)\omega_{k}\left(\sum_{i=1}^{k}c_{i}\right)\operatorname{Vol}(X)\\
    &=&\omega_{k}\operatorname{Vol}(X)m(g).
  \end{eqnarray*}
  Combining the discussion above we arrive at the equation
   \begin{eqnarray*}
     \sum_{i}\int_{M}\langle\nabla dy^{i},\nabla dy^{i}\rangle d\operatorname{vol}_{g}&=&\omega_{k}\operatorname{Vol}(X)m(g).
  \end{eqnarray*}
  In particular $m(g)=0$ implies that $dy^{i}$ is parallel.
  
  Therefore $dy^{i}$ is parallel for $1\leq  i\leq k$, and as they are approaching orthonormal at infinity, they are exactly orthonormal on $M$. 
  Consider the map $F:M^{n}\to \mathbb{R}^{k}$, $F(p)=\{y_{1}(p),\cdots,y_{k}(p)\}$. Let $\Phi^{i}_{t}$ be the flow of $\nabla y_{i}$, which is a complete vector field since it has norm one. Since $\nabla y_{i}$ is parallel, it is a killing vector, i.e. the flow $\Phi^{i}_{t}$ action on $M^{n}$ is isometric. Therefore, for $q\in \mathbb{R}^{k}$,  $\Phi^{i}_{t}:F^{-1}(q)\to F^{-1}(q+t e_{i})$ is isometric and so is $\Phi^{1}_{y^1}\circ\cdots\circ\Phi^{k}_{y^k}:F^{-1}(q)\to F^{-1}(q+y)$. Since $y^{i}, 1\leq i \leq k$, form an asymptotic coordinate system for the $\mathbb R^k$ factor and $g$ is asymptotic to the product metric, 
  letting $y\to \infty$, we find an isometry $F^{-1}(q)\cong X$. Thus $M^{n}$ is a fiber bundle over $\mathbb{R}^{k}$ with fibers isometric to $X$. Since $\mathbb{R}^{k}$ is contractible, the bundle must be trivial. On the other hand, the de Rham Decomposition Theorem says that this is a local metric product. Therefore $(M^{n},g)\cong (\mathbb{R}^{k}\times X,g_{\mathbb{R}^{k}}+g_{X})$.

 From the proof of the rigid part, as a generalization of Proposition 4.12 in \cite{CLSZ}, we have 
    \begin{proposition}
      For a complete noncompact smooth manifold $(M^{n},g)$ asymptotic $(\mathbb{R}^{k}\times X^{n-k},\hat{g}=g_{\mathbb{R}^{k}}+g_{x})$, if $|\operatorname{ Ric}|_{\hat{g}}+r|\nabla \operatorname{ Ric}|_{\hat{g}}=O(r^{-k-\epsilon})$ for some $\epsilon>0$ then 
      $$m^{GB}(g)=c\,m(g)$$
      for some $c> 0$.
    \end{proposition}
    
    It is very interesting that this notion of the Gauss-Bonnet mass plays a crucial role in the rigidity part of the Positive Mass Theorem.

\end{document}